\documentclass[a4paper, 12pt, oneside]{article}
\usepackage[usenames, dvipsnames]{xcolor}
\pdfoutput=1 
\usepackage[T1]{fontenc}
\usepackage[british]{babel}
\usepackage{lmodern}
\usepackage{etoolbox}
\usepackage{amsmath, amsthm, amssymb, bm, mleftright}
\usepackage{booktabs, longtable, caption, subcaption, multirow}
\usepackage[space]{grffile}
\usepackage[margin=2.5cm, footskip=1cm]{geometry}
\usepackage{enumitem}
\usepackage[protrusion=allmath]{microtype}
\usepackage{tikz}
\usetikzlibrary{cd, calc}
\usepackage{listings}
\usepackage[colorlinks, citecolor=blue, linkcolor=blue, linktoc=section]{hyperref}

\allowdisplaybreaks
\usepackage[capitalise, noabbrev]{cleveref}
\creflabelformat{enumi}{(#2#1#3)}
\creflabelformat{enumii}{(#2#1#3)}
\creflabelformat{enumiii}{(#2#1#3)}

\usepackage{titlesec}
\titleformat*{\section}{\large\bfseries}
\titleformat*{\subsection}{\normalsize\bfseries}
\newlength{\VerticalSpaceAfterParagraph}
\titlespacing*{\paragraph}{0pt}{\VerticalSpaceAfterParagraph}{1em}

\usepackage{titling}
\pretitle{\vspace{-\baselineskip}\begin{center}\Large\bfseries}
\posttitle{\end{center}\vspace{-0.25\baselineskip}}
\preauthor{\begin{center}}
\postauthor{\end{center}}
\predate{\begin{center}}
\postdate{\end{center}}

\setlist
  {
    topsep = 5.0pt plus 2.0pt minus 3.0pt,
    partopsep = 1.5pt plus 1.0pt minus 1.0pt,
    parsep = 2.5pt plus 1.25pt minus 0.5pt,
    itemsep = 0pt plus 1.25pt minus 0.5pt
  }

\theoremstyle{plain}
\newtheorem{theorem}{Theorem}
\newtheorem*{theorem*}{Theorem}
\newtheorem{proposition}[theorem]{Proposition}
\newtheorem{lemma}[theorem]{Lemma}

\theoremstyle{definition}
\newtheorem{definition}[theorem]{Definition}
\newtheorem{example}[theorem]{Example}

\theoremstyle{remark}
\newtheorem{remark}[theorem]{Remark}

\numberwithin{theorem}{section}
\numberwithin{table}{section}
\numberwithin{equation}{theorem}

\usepgfmodule{oo}

\DeclareRobustCommand\ShowAuthors[2]{%
  \ShowAuthorsSignal.emit({#1},{#2})%
}

\pgfoonew \ShowAuthorsSignal=new signal()

\DeclareRobustCommand\ShowAffiliations[1]{%
  \ShowAffiliationsSignal.emit({#1})%
}

\pgfoonew \ShowAffiliationsSignal=new signal()

\newcommand\Author[1]{
  \pgfoonew \CurrentPerson=new person()
  \CurrentPerson.set author(#1)
}

\newcommand\Email[1]{
  \CurrentPerson.set email(#1)
}

\newcommand\Address[1]{
  \CurrentPerson.set address(#1)
}

\newcommand\FirstPerson{0}
\newcommand\LastPerson{}

\pgfooclass{person}{
  \attribute author;
  \attribute email;
  \attribute address;

  \method person() {
    \pgfoothis.get id(\theid)

    \ifnum \FirstPerson=0
      \edef\FirstPerson{\theid}
    \fi

    \edef\LastPerson{\theid}

    \pgfoothis.get handle(\me)
    \ShowAuthorsSignal.connect(\me,show author)
    \ShowAffiliationsSignal.connect(\me,show affiliation)
  }

  \method get author(#1) {
    \pgfooget{author}{#1}
  }

  \method set author(#1) {
    \pgfooset{author}{#1}
  }

  \method get email(#1) {
    \pgfooget{email}{#1}
  }

  \method set email(#1) {
    \pgfooset{email}{#1}
  }

  \method get address(#1) {
    \pgfooget{address}{#1}
  }

  \method set address(#1) {
    \pgfooset{address}{#1}
  }

  \method show author(#1,#2) {%
    \pgfoothis.get id(\theid)%
    \pgfooget{author}{\theauthor}%
    \ifnum\theid=\FirstPerson%
    \else%
      \ifnum\theid=\LastPerson%
        #2%
      \else%
        #1%
      \fi%
    \fi%
    \mbox{\theauthor}%
  }

  \method show affiliation(#1,#2) {%
    \pgfoothis.get id(\theid)%
    \pgfooget{author}{\theauthor}%
    \pgfooget{email}{\theemail}%
    \pgfooget{address}{\theaddress}%
    \ifnum\theid=\FirstPerson%
      #1%
    \else%
      #2%
    \fi%
    \noindent\begin{minipage}{\linewidth}
      \noindent\begin{tabular}[t]{@{}l}
        \theauthor \quad \texttt{\theemail}\\
      \end{tabular}\\
      \theaddress%
    \end{minipage}%
  }
}

\newcommand\blfootnote[1]
  {%
    \begingroup%
      \renewcommand\thefootnote{}%
      \footnote{#1}%
      \addtocounter{footnote}{-1}%
    \endgroup%
  }

\newcommand*\abs[1]{\left\lvert #1 \right\rvert}

\usepackage{mathtools}
\usepackage{calc}

\Author{Erik Paemurru}
\Address{Mathematik und Informatik, Universität des Saarlandes, 66123 Saarbrücken, Germany}
\Email{paemurru@math.uni-sb.de}

\title{Reading the log canonical threshold of a plane curve singularity from its Newton polyhedron}
\author{\ShowAuthors{, }{, }}
\date{26th~April 2024}
\newcommand\keywords{Complex singularity exponent, complex oscillation index, Newton polygon, remoteness}
\newcommand\subjclass{14H20 (Primary) 14B05, 32S25, 14Q05, 14H50}

\begin{document}

\maketitle

\begin{abstract}
There is a proposition due to Kollár 1997 on computing log canonical thresholds of certain hypersurface germs using weighted blowups, which we extend to weighted blowups with non-negative weights. Using this, we show that the log canonical threshold of a convergent complex power series is at most $1/c$, where $(c, \ldots, c)$ is a point on a facet of its Newton polyhedron. Moreover, in the case $n = 2$, if the power series is weakly normalised with respect to this facet or the point $(c, c)$ belongs to two facets, then we have equality. This generalises a theorem of Varchenko 1982 to non-isolated singularities.
\blfootnote{\textup{2020} \textit{Mathematics Subject Classification}. \subjclass{}.}
\end{abstract}

\tableofcontents

\section*{Declarations}

\paragraph{Availability of data and material.} There is no supporting data.

\paragraph{Competing interests.} There are no competing interests.

\paragraph{Funding.} This is a contribution to Project-ID 286237555 -- TRR 195 -- by the Deutsche Forschungsgemeinschaft (DFG, German Research Foundation).

\paragraph{Authors' contributions.} The article is independently authored by Erik Paemurru.

\paragraph{Acknowledgements.} I would like to thank Janko Böhm and Magdaleen S. Marais for discussions on \cite{BMP20}.

\section{Introduction}

Let $f$ be a holomorphic function, not identically zero, on a domain of $\mathbb C^n$ where $n$ is any positive integer and let $P$ be a zero of~$f$. The log canonical threshold of $f$ at $P$ can be characterised in many ways:
\begin{enumerate}[label=(\arabic*), ref=\arabic*]
\item it is the greatest positive rational number $\lambda$ such that the pair $(\mathbb C^2, \lambda V(f))$ is log canonical (\cref{def:lct}),
\item it is the supremum of rational numbers $\lambda$ such that $\abs{f}^{-\lambda}$ is $L^2$ in a neighbourhood of~$P$, sometimes called the \emph{complex singularity exponent} (\cite[Definition~8.4]{KK13}),
\item it is the smallest \emph{jumping number} of $V(f)$ (\cite[Definition~9.3.22]{Laz04b}),
\item it is the negative of the largest zero of the Bernstein--Sato polynomial of~$f$ (\cite[Theorem~10.6]{Kol97}),
\item if $f$ defines an isolated singularity, then the log canonical threshold is equal to $\min(1, \beta_\mathbb C(f))$ where $\beta_\mathbb C(f)$ is the \emph{complex singular index} (see \cite[Section~4]{Var82} or \cite[Theorem~9.5]{Kol97}), called the \emph{complex oscillation index} in \cite[13.1.5]{AGZV12b} and the \emph{complex singularity index} in \cite{Ste85},
\item if $f$ defines an isolated singularity, then the complex singular index minus one coincides with the smallest \emph{spectrum number} (\cite[p~558]{Ste85}).
\end{enumerate}
If $f$ does not define an isolated singularity, then it is not clear whether the log canonical threshold and the complex singular index coincide (\cite[9.6]{Kol97}).

Considering $f$ as a power series in $\mathbb C\{x_1, \ldots, x_n\}$, the Newton polyhedron (\cref{def:Newton polyhedron}) is the convex hull of all the points $(i_1, \ldots, i_n) \in \mathbb R^n$ such that the monomial $x_1^{i_1} \cdot \ldots \cdot x_n^{i_n}$ has a non-zero coefficient in some power series $g$ in the ideal generated by~$(f)$. The \emph{distance to the Newton polyhedron} is defined as the rational number $c$ such that the point $(c, \ldots, c)$ is on the boundary of the Newton polyhedron of~$f$.

In case of isolated plane curve singularities, by \cite[Theorem~4.4]{Var82} there is a coordinate change (described in \cite{Var77} for isolated real plane curve singularities) such that the log canonical threshold is equal to the reciprocal of the distance. The proof is analytic in nature. See \cite[Theorem~6.40]{KSC04} for an algebraic proof. Note that the proof of \cite[Theorem~6.40]{KSC04} works only for reduced plane curves since the last sentence of the proof uses finite determinacy of isolated singularities.

The main result of this paper is extending \cite[Theorem~4.4]{Var82} to non-isolated plane curve singularities:

\newcommand\MainTheorem{
  Let $f \in \mathbb C\{x_1, \ldots, x_n\}$ be any non-zero power series satisfying $f(\bm 0) = 0$. Let $c \in \mathbb Q_{>0}$ be the unique number such that a (not necessarily compact) facet $\Delta$ of the Newton polyhedron of $f$ contains the point $(c, \ldots, c) \in \mathbb R^n$. Then, $\operatorname{lct}_{\bm 0}(f) \leq 1/c$. Moreover, in the case $n = 2$, if $f$ is weakly normalised with respect to~$\Delta$ or the point $(c, c)$ is in the intersection of two facets, then $\operatorname{lct}_{\bm 0}(f) = 1/c$.
}
\begin{theorem*}[= {\cref{thm:reading lct}}] \label{thm:main}
\MainTheorem{}
\end{theorem*}

The proof relies on \cref{thm:weighted formula power series}, which is a generalisation of \cite[Proposition~8.14]{Kol97} to weighted blowups with non-negative weights.

A generalisation of \cite[Algorithm~4]{BMP20} shows how to find a coordinate change such that $f \in \mathbb C\{x, y\}$ is normalised with respect to every facet of the Newton polyhedron (\cref{thm:normalised Newton polyhedron}). To read off the log canonical threshold, a weaker condition is enough, which we call \emph{weakly normalised} (\cref{def:normalised}). In case the singularity is non-isolated, \cref{thm:normalised Newton polyhedron} only gives a formal coordinate system, which is enough for computing the log canonical threshold (\cref{rem:algorithm}\labelcref{itm:algorithm up to arbitrary precision}). In case the principal part of $f \in \mathbb C\{x, y\}$ is non-degenerate, by \cite[Remark~3.28]{BMP20}, all power series that are right equivalent to $f$ and are normalised with respect to every compact facet of their Newton polyhedrons have the same Newton polyhedron. This does not always hold without the non-degeneracy assumption as shown by \cref{exa:two normalised Newton polyhedrons,exa:minimal Newton polyhedrons}.

The log canonical threshold at a point of a reduced plane curve can also be computed using Puiseux expansions. An explicit formula is given in \cite[Theorem~2.13]{GHM16}, which generalises \cite[Theorems 1.2 and~1.3]{Kuw99}, that depends only on the first two maximal contact values of the branches and their intersection numbers.

Lastly, we mention some similar results that have been proved in higher dimensions, assuming the power series has enough monomials and with general enough coefficients. More precisely, we say that a power series $f \in \mathbb C\{x_1, \ldots, x_n\}$ is \emph{non-degenerate} with respect to a face $\sigma$ of its Newton polyhedron if the polynomials
\[
\frac{\partial f_\sigma}{\partial x_1}, \ldots, \frac{\partial f_\sigma}{\partial x_n}
\]
do not have common zeros in $(\mathbb C \setminus \{0\})^n$, where $f_\sigma$ is the sum of the terms of $f$ lying on~$\sigma$. We say that $f$ has \emph{non-degenerate principal part} if $f$ is non-degenerate with respect to all compact faces (\cite[Definition~6.2.2]{AGZV12b}). As a stronger condition, we say that $f$ is \emph{non-degenerate} if $f$ is non-degenerate with respect to all faces (\cite[Definition~5]{How03}). In particular, if $f$ is non-degenerate, then $f$ has to be non-singular outside the coordinate hyperplanes.

By \cite[Theorem~13.2(1)]{AGZV12b}, if $f \in \mathbb C\{x_1, \ldots, x_n\}$ defines an isolated singularity and has non-degenerate principal part, then the log canonical threshold of $f$ is at most the reciprocal of the distance of the Newton polyhedron, with equality if the distance is greater than~$1$. The proof is analytic.

If $f$ is a non-degenerate polynomial, then we can combinatorially compute all the jumping numbers, not only the log canonical threshold, and in fact describe the multiplier ideal of the divisor defined by~$f$. Namely, let $\mathfrak a_f$ denote the ideal generated by the monomials appearing with a non-zero coefficient in~$f$.
By \cite[Main Theorem]{How01} (also given in \cite[Theorem~9.3.27]{Laz04b}), the multiplier ideal $\mathcal J(r \cdot \mathfrak a)$ of a monomial ideal $\mathfrak a \subseteq \mathbb C[x_1, \ldots, x_n]$ coincides with the ideal generated by all monomials $x_1^{i_1} \cdot \ldots \cdot x_n^{i_n}$ such that $(i_1, \ldots, i_n) + (1, \ldots, 1)$ is in the interior of~$rP$, where $rP$ is the homothety with factor $r$ applied to the the Newton polyhedron $P$ of~$\mathfrak a$.
In the preprint \cite[Theorem~12]{How03}, it is proved that if $f$ is non-degenerate, then a toric log resolution of $\mathfrak a_f$ also log-resolves $\operatorname{div} f$. In this case, by \cite[Corollary~13]{How03} (stated also in \cite[Theorem~9.3.37]{Laz04b}), for every rational number~$0 < r < 1$, we have the equality of multiplier ideals,
\[
\mathcal J(r \cdot \operatorname{div} f) = \mathcal J(r \cdot \mathfrak a_f).
\]
As detailed in \cite[Example~9.3.31]{Laz04b}, this recovers the result on log canonical thresholds of \cite[Theorem~13.2(1)]{AGZV12b} in the special case of non-degenerate polynomials.

\section{Preliminaries}

\subsection{Power series and Newton polyhedron}

\begin{definition}
The field of complex numbers is denoted~$\mathbb C$.
The $\mathbb C$-algebra of power series in variables $x_1, \ldots, x_n$ that are absolutely convergent in a neighbourhood of $\bm 0$ is denoted by $\mathbb C\{x_1, \ldots, x_n\}$, where $n$ is a positive integer.
The (possibly non-reduced) complex subspace of $\mathbb C^n$ defined by a convergent power series $f \in \mathbb C\{x_1, \ldots, x_n\}$ is denoted by $V(f)$.
The \textbf{saturation} of a power series $f \in \mathbb C[[x_1, \ldots, x_n]]$, denoted $\operatorname{sat}(f)$, is the power series $f / (x_1^{a_1} \cdot \ldots \cdot x_n^{a_n})$ where $a_1, \ldots, a_n \in \mathbb Z_{\geq0}$ are as large as possible.

Two formal power series $f, g \in \mathbb C[[x_1, \ldots, x_n]]$ are \textbf{formally right equivalent} if there exists a $\mathbb C$-algebra isomorphism $\Phi$ of $\mathbb C[[x_1, \ldots, x_n]]$ such that $\Phi(f) = g$. Two convergent power series $f, g \in \mathbb C\{x_1, \ldots, x_n\}$ are \textbf{right equivalent} if there exists a $\mathbb C$-algebra isomorphism $\Phi$ of $\mathbb C\{x_1, \ldots, x_n\}$ such that $\Phi(f) = g$ (\cite[Definition~I.2.9]{GLS07}).

Let $\bm w = (w_1, \ldots, w_n) \in \mathbb Q_{\geq0}^n$ be non-negative rational numbers, not all zero, and let
\[
f := \sum_{\crampedclap{i_1, \ldots, i_n \in \mathbb Z_{\geq0}}} a_{i_1, \ldots, i_n} x_1^{i_1} \cdot \ldots \cdot x_n^{i_n}
\]
be a power series in $\mathbb C[[x_1, \ldots, x_n]]$, where $a_{i_1, \ldots, i_n}$ are complex numbers. The \textbf{$\bm w$-weight} of~$f$, denoted $\operatorname{wt}_{\bm w}(f)$, is defined to be
\[
\operatorname{wt}_{\bm w}(f) := \min \mleft\{ i_1 w_1 + \ldots + i_n w_n \;\middle|\; \begin{aligned}
  & i_1, \ldots, i_n \in \mathbb Z_{\geq0},~ \text{the coefficient}\\
  & \text{of } x_1^{i_1} \cdot \ldots \cdot x_n^{i_n} \text{ is non-zero in~$f$}
\end{aligned} \mright\} \in \mathbb Q_{\geq0} \cup \{\infty\}.
\]
If $f$ is not zero, then the \textbf{$\bm w$-weighted-homogeneous leading term} of~$f$, denoted $f_{\bm w}$, is defined to be
\[
\sum_{\crampedclap{\substack{i_1, \ldots, i_n \in \mathbb Z_{\geq0}\\i_1 w_1 + \ldots + i_n w_n = \operatorname{wt}_{\bm w}(f)}}} a_{i_1, \ldots, i_n} x^{i_1} \cdot \ldots \cdot x^{i_n}.
\]
\end{definition}

\begin{definition}[{\cite[Section~12.7]{AGZV12a}, \emph{local Newton polytope} in \cite[Definition~2.14]{GLS07}}] \label{def:Newton polyhedron}
The \textbf{Newton polyhedron} of $f \in \mathbb C[[x_1, \ldots, x_n]]$ is the subset of $\mathbb R^n$ given by the convex hull of the union of the subsets $(i_1, \ldots, i_n) + (\mathbb R_{\geq0})^n$ taken over non-negative integers $i_1, \ldots, i_n$ such that $x_1^{i_1} \cdot \ldots \cdot x_n^{i_n}$ has a non-zero coefficient in~$f$.
\end{definition}

\begin{definition}[{\cite[Definition~I.2.15]{GLS07}}]
Let $f \in \mathbb C[[x_1, \ldots, x_n]]$ be a power series such that its Newton polyhedron has non-empty intersection with every coordinate axis. We say that $f$ is \textbf{Newton non-degenerate} or has \textbf{non-degenerate Newton boundary} if for every compact facet $\Delta$ of its Newton polyhedron and non-zero normal vector $\bm w \in \mathbb Z_{\geq0}^2$ of~$\Delta$, $\operatorname{sat}(f_{\bm w})$ defines an isolated singularity.
\end{definition}

\subsection{Log canonical threshold and weighted blowups}

\begin{definition}[{\cite[Notation~0.4]{KM98}}]
Let $X$ be a smooth complex space. A \textbf{$\mathbb Q$-divisor} on $X$ is a formal $\mathbb Q$-linear combination $\sum \lambda_i D_i$ of prime divisors $D_i$ where $\lambda_i \in \mathbb Q$. A $\mathbb Z$-divisor $\sum D_i$ is \textbf{snc} if all the prime divisors $D_i$ are smooth and around every point of~$X$, $\sum D_i$ is locally given by $V(x_1 \cdot \ldots \cdot x_k)$ in $\mathbb C^n$ where $0 \leq k \leq n$.

Let $D$ be a $\mathbb Q$-divisor on a smooth complex space~$X$. A \textbf{log resolution of $(X, D)$} is a
proper modification (\cite[Definition VII.1.1]{CDG+94}) $\pi\colon X' \to X$ from a smooth complex space $X'$
such that the exceptional locus $E$ of $\pi$ is of pure codimension $1$ and $E \cup \pi^{-1}(\operatorname{Supp}(D))$ is snc.

For any proper bimeromorphic holomorphism $\varphi\colon X' \to X$ from a smooth complex space~$X'$, the \textbf{relative canonical divisor} of~$\varphi$, denoted~$K_\varphi$, is the unique $\mathbb Q$-divisor that is linearly equivalent to $\varphi^*(K_X) - K_{X'}$ and supported on the exceptional locus of~$\varphi$, where $K_X$ and $K_{X'}$ are the canonical classes of respectively $X$ and~$X'$. The \textbf{log pullback} of $D$ with respect to $\varphi$ is the $\mathbb Q$-divisor $D' = K_\varphi + \varphi^* D$ on~$X'$. We have the $\mathbb Q$-linear equivalence $K_{X'} + D' \sim \varphi^*(K_X + D)$.
\end{definition}

\begin{definition}[{\cite[Definition 3.5]{Kol97} or \cite[Definition~2.34]{KM98}}] \label{def:lct}
Let $D$ be a $\mathbb Q$-divisor on a smooth complex space $X$ and let $P \in X$ be a point. The pair $(X, D)$ is \textbf{log canonical at~$P$} if we can restrict $(X, D)$ to an open neighbourhood of $P$ such that there exists a log resolution with the coefficient of every prime divisor in the log pullback of $D$ at most~$1$. The \textbf{log canonical threshold of $(X, D)$ at~$P$} is
\[
\operatorname{lct}_P(X, D) := \operatorname{sup}\bigl\{ \lambda \in \mathbb Q_{>0} \;\big|\; (X, \lambda D) \text{ is log canonical at~$P$} \bigr\}.
\]
If $f \in \mathbb C\{x_1, \ldots, x_n\}$ is not zero, then $\operatorname{lct}_{\bm 0}(f)$ is the log canonical threshold of $(\mathbb C^n, V(f))$ at the origin, where $\mathbb C^n$ and $V(f)$ are considered as complex spaces defined around the origin.
\end{definition}

\begin{proposition}[{\cite[Proposition~8.19]{Kol97}}] \label{thm:approximation}
Let $f \in \mathbb C\{x_1, \ldots, x_n\}$ be non-zero and satisfy $f(\bm 0) = 0$. Let $d$ be a non-negative integer such that the truncation $f_{\leq d}$ of $f$ up to degree~$d$ is non-zero. Then,
\[
\abs{\operatorname{lct}_{\bm 0}(f) - \operatorname{lct}_{\bm 0}(f_{\leq d})} \leq \frac{n}{d+1}.
\]
\end{proposition}

\begin{remark}
See \cite[Proposition-Definition~10.3]{KM92} for a toric description and \cite[Definition~4.56]{KM98} for an algebraic description of weighted blowups of affine space with positive integer weights.

If the variables of an affine space $\mathbb A^k$ over $\mathbb C$ have non-negative integer weights, not all zero, then write $\mathbb A^k$ as a product $\mathbb A^m \times \mathbb A^n$ where all the weights of $\mathbb A^m$ are zero and all the weights of $\mathbb A^n$ are positive. The weighted blowup of $\mathbb A^k$ is the morphism $\mathrm{id}_{\mathbb A^m} \times \varphi$ where $\varphi$ is the weighted blowup of $\mathbb A^n$.

If $X$ is a complex subspace of~$\mathbb C^k$, then the weighted blowup of $X$ is the induced holomorphism from strict transform of $X$ with respect to the analytification of the weighted blowup of $\mathbb A^k$.
\end{remark}

\begin{lemma} \label{thm:discrepancy}
Let $\varphi\colon W \to \mathbb C^n$ be a weighted blowup of $\mathbb C^n$ with non-negative weights $\bm w = (w_1, \ldots, w_n)$ satisfying $\gcd(w_1, \ldots, w_n) = 1$. Let $E$ be the exceptional prime divisor. Let $f \in \mathbb C\{x_1, \ldots, x_n\}$ be any non-zero power series. Let $\lambda$ be any rational number. Then, the coefficient of $E$ in the log pullback of $\lambda V(f)$ is $1 + \lambda \operatorname{wt}_{\bm w}(f) - \sum w_i$.
\end{lemma}

\begin{proof}
Using toric geometry (\cite[Lemma~11.4.10]{CLS11}), we find
\[
K_W = \varphi^* K_{\mathbb C^n} + \bigl(\sum w_i - 1\bigr) E.
\]
Let $v$ be the discrete valuation of the field of meromorphic functions of $\mathbb C^n$ given by the order of vanishing along the exceptional divisor~$E$. Then, $v(f) = \operatorname{wt}_{\bm w}(f)$. Let $V(f) = \sum \alpha_j C_j$, where $C_j$ are prime divisors and the coefficients $\alpha_j$ are non-negative integers. Then,
\[
\sum \alpha_j (\varphi_*^{-1} C_j) = \varphi^* V(f) - \operatorname{wt}_{\bm w}(f) E.
\]
The \lcnamecref{thm:discrepancy} follows.
\end{proof}

\section{Main} \label{sec:main}

\subsection{Computing the log canonical threshold using weights} \label{sec:lct from weights}

In this \lcnamecref{sec:lct from weights}, we prove \cref{thm:weighted formula power series}, which is the main tool used in \cref{sec:lct from Newton polyhedron} for computing log canonical thresholds.

\begin{remark}
A version of \cref{thm:weighted formula power series} is proved in \cite[Proposition~8.14]{Kol97} with two differences:
\begin{enumerate}[label=(\arabic*), ref=\arabic*]
\item in \cite[Proposition~8.14]{Kol97}, the weights are required to be positive, whereas we allow any non-negative weights, not all zero,
\item the condition that $(\mathbb C^n, b \cdot V(f_{\bm w}))$ is log canonical is replaced by the condition that $(\mathbb P^{n-1}, V(f_{\bm w}(x_1^{w_1}, \ldots, x_n^{w_n})))$ is log canonical, where $\mathbb P^{n-1}$ is the exceptional divisor of the blowup of $\mathbb C^n$ at~$\bm 0$.
\end{enumerate}
The statement in \cite[Proposition~8.14]{Kol97} contains an error, it should instead say that $(\mathbb P^{n-1}, b \cdot V(f_{\bm w}(x_1^{w_1}, \ldots, x_n^{w_n})))$ is log canonical, otherwise \cref{exa:counter-example smooth,exa:counter-example Ak} are counter-examples.

\Cref{thm:weighted formula power series} is stated, without proof, for positive weights in \cite[Proposition~2.1]{Kuw99}.
\end{remark}

\begin{proposition} \label{thm:weighted formula power series}
Let $n$ be a positive integer. Let $f \in \mathbb C\{x_1, \ldots, x_n\}$ be a non-zero power series. Assign non-negative rational weights $\bm w = (w_1, \ldots, w_n)$ to the variables, not all zero. Let $f_{\bm w}$ denote the $\bm w$-weighted-homogeneous leading term of~$f$. Define $b := \sum_i w_i / \operatorname{wt}_{\bm w}(f) \in \mathbb Q_{>0} \cup \{\infty\}$. Define the subset
\[
C := V_{\mathbb C^n}(\{x_i \mid i \in \{1, \ldots, n\},\, w_i > 0\}).
\]
Then, $\operatorname{lct}_{\bm 0}(f) \leq b$.
Moreover, considering $\mathbb C^n$ and $V(f_{\bm w})$ as complex spaces defined in a neighbourhood of~$\bm 0$, if $b$ is finite and $(\mathbb C^n, b \cdot V(f_{\bm w}))$ is log canonical outside~$C$, then $\operatorname{lct}_{\bm 0}(f) = b$.
\end{proposition}

\begin{proof}
After scaling by a suitable positive rational number, the numbers $w_1, \ldots, w_n$ are non-negative integers with $\gcd(w_1, \ldots, w_n) = 1$.

If there is exactly one index $i_0 \in \{1, \ldots, n\}$ such that $w_{i_0}$ is positive, then $\operatorname{wt}(f) = 1/b$ is the coefficient of the prime divisor $V(x_{i_0})$ in the divisor~$V(f)$. Therefore, $\operatorname{lct}_{\bm 0}(f) \leq b$. Now, let $b$ be finite and let $(\mathbb C^n, b \cdot V(f_{\bm w}))$ be log canonical away from $V(x_{i_0})$. The complex space $V_{\mathbb C^n}(f_{\bm w}) \setminus V_{\mathbb C^n}(x_{i_0})$ is biholomorphic to $(\mathbb C^1 \setminus \{0\}) \times V_{\mathbb C^{n-1}}(f_{\bm w} / x_{i_0}^{1/b})$. Therefore, $(\mathbb C^{n-1}, b \cdot V(f_{\bm w} / x_{i_0}^{1/b}))$ is log canonical. Note that $V(x_{i_0}) + b \cdot V(f / x_{i_0}^{1/b})$ is precisely the divisor $b \cdot V(f)$ and that the restriction of $b \cdot V(f / x_{i_0}^{1/b})$ to $V(x_{i_0})$ is $b \cdot V(f_{\bm w} / x_{i_0}^{1/b})$. By inversion of adjunction (\cite[Theorem~7.5]{Kol97}), $(\mathbb C^n, b \cdot V(f))$ is log canonical. Therefore, $\operatorname{lct}_{\bm 0}(f) = b$.

Below we consider the case where there are at least two positive weights among~$\bm w$. Let $\varphi\colon W \to \mathbb C^n$ be the $\bm w$-weighted blowup and $E$ its exceptional divisor. For every $\lambda \in \mathbb Q_{>0}$, by \cref{thm:discrepancy}, the coefficient of $E$ in the log pullback of $\lambda V(f)$ is
\[
1 + \lambda \operatorname{wt}_{\bm w}(f) - \sum w_i.
\]
Therefore, $\operatorname{lct}_{\bm 0}(f) \leq b$.

Now, let $b$ be finite and let $(\mathbb C^n, b \cdot V(f_{\bm w}))$ be log canonical outside~$C$. We show that $(E, b \cdot V(f_{\bm w}))$ is log canonical. For every $i \in \{1, \ldots, n\}$ such that $w_i$ is positive, let the group $\mu_{w_i}$ of $w_i$-th roots of unity act on $\mathbb C^n$ by $\xi \cdot (x_1, \ldots, x_n) = (\xi^{w_1} x_1, \ldots, \xi^{w_n} x_n)$, where $\xi$ is a primitive $w_i$-th root of unity. Let
\[
p\colon \mathbb C^n \setminus V(x_i) \to \frac{\mathbb C^n \setminus V(x_i)}{\mu_{w_i}}
\]
be the natural holomorphism. By \cite[Theorem~8.12]{Kol97},
\[
\mleft(\frac{\mathbb C^n \setminus V(x_i)}{\mu_{w_i}}, V(f_{\bm w})\mright)
\]
is log canonical. For every monomial $M$ in the $\mu_{w_i}$-invariant $\mathbb C$-algebra $\mathbb C[x_0, \ldots, x_n, x_i^{-1}]^{\mu_{w_i}}$, there exists an integer $k_M$ such that $\operatorname{wt}_{\bm w}(M) = k_M w_i$. The $\mathbb C$-algebra isomorphism
\[
\begin{aligned}
  \mathbb C[x_0, \ldots, x_n, x_i^{-1}]^{\mu_{w_i}} & \to \mathbb C[x_1, \ldots, x_n, x_i^{-1}]_0 \otimes_{\mathbb C} \mathbb C[y, y^{-1}]\\
  x_i & \mapsto y\\
  M & \mapsto y^{k_M} x_i^{-k_M} M
\end{aligned}
\]
induces a biholomorphism
\[
\frac{\mathbb C^n \setminus V(x_i)}{\mu_{w_i}} \cong (E \setminus V(x_i)) \times (\mathbb C^1 \setminus \{0\})
\]
which takes $V(f_{\bm w}) \subseteq (\mathbb C^n \setminus V(x_i))/(\mu_{w_i})$ to $V_{E \setminus V(x_i)}(f_{\bm w}) \times (\mathbb C^1 \setminus \{\mathrm{pt}\})$. Since log canonicity is a local analytic property (\cite[Proposition~4.4.4]{Mat02}) and since taking a product with a smooth complex space does not change discrepancies, $(E, b \cdot V(f_{\bm w}))$ is log canonical.

We show that $\operatorname{lct}_{\bm 0}(\mathbb C^n, V(f)) = b$. Since the exceptional divisor $E$ of the weighted blowup $\varphi\colon W \to \mathbb C^n$ is the product of an affine space and a weighted projective space, $E$ is log terminal. Let $D_{\bm w}'$ be the log pullback of~$b \cdot V(f_{\bm w})$. By inversion of adjunction (\cite[Theorem~7.5]{Kol97}), since $(E, b \cdot V(f_{\bm w}))$ is log canonical, $(W, D_{\bm w}')$ is log canonical near~$E$. By \cite[Lemma~3.10.2]{Kol97}, $(\mathbb C^n, b \cdot V(f))$ is log canonical at~$\bm 0$. Therefore, $\operatorname{lct}_{\bm 0}(f) = b$.
\end{proof}

\cref{exa:counter-example smooth,exa:counter-example Ak} show that the number $b$ in \cref{thm:weighted formula power series} can be greater than~$1$.

\begin{example} \label{exa:counter-example smooth}
Let $f := x + y^d$, where $d$ is any positive integer. Since $f$ is smooth at $\bm 0$, $\operatorname{lct}_{\bm 0}(f) = 1$. On the other hand, choosing weights $(1, 1/d)$, we find $b = (d+1)/d > 1$.
\end{example}

\begin{example} \label{exa:counter-example Ak}
Let $f = x_1^2 + \ldots + x_{n-1}^2 + x_n^{k+1}$, where $n \geq 3$ and $k \geq 1$. If $n = 3$, then $f$ defines a canonical surface singularity, and if $n \geq 4$, then $f$ defines a terminal $(n-1)$-fold singularity. By inversion of adjunction (\cite[Theorem~7.5]{Kol97}), $\operatorname{lct}_{\bm 0}(f) = 1$. On the other hand, choosing weights $(1/2, \ldots, 1/2, 1/(k+1))$, we find $b = (n-1)/2 + 1/(k+1) > 1$.
\end{example}

\subsection{Choosing a coordinate system}

In most coordinate systems, the Newton polyhedron gives very little information about the power series. Namely, for every non-zero non-unit power series, after a general linear coordinate change, its Newton polyhedron has exactly one compact facet, and that compact facet has a normal vector $(1, \ldots, 1)$. \Cref{thm:normalised Newton polyhedron} describes a natural coordinate change such that the Newton polyhedron is interesting for our purposes. In this case, we can read off the log canonical threshold from the Newton polyhedron by \cref{thm:reading lct}.

\Cref{def:normalised} is \cite[Definition~3.24]{BMP20} generalised to the situation where the power series is not necessarily right equivalent to a Newton non-degenerate power series. It coincides with \cite[Definition~3.24]{BMP20} if the power series is right equivalent to a Newton non-degenerate one. In \cref{thm:reading lct}, we only need a weaker condition which we call \emph{weakly normalised with respect to~$\Delta$}.

\begin{definition} \label{def:normalised}
Let $f \in \mathbb C[[x, y]]$ be non-zero. Let $\Delta$ be a (not necessarily compact) facet of the Newton polyhedron of~$f$. Let $\bm w = (w_x, w_y)$ be the unique non-negative integers with $\gcd(w_x, w_y) = 1$ such that $\bm w$ is normal vector of~$\Delta$. Let $f_{\bm w}$ be the $\bm w$-weighted homogeneous leading term of~$f$. Let $a, b$ be the greatest non-negative integers such that $f_{\bm w} / (x^a y^b)$ is a power series. Let $d$ be the greatest non-negative integer such that an irreducible power series to the power $d$ divides $\operatorname{sat}(f_{\bm w})$. We say that $f$ is \textbf{weakly normalised with respect to~$\Delta$} if one of the following holds:
\begin{enumerate}[label=(W\arabic*), ref=W\arabic*, leftmargin=3.5em]
\item $w_x = 0$ or $w_y = 0$,
\item $d \leq \max(a, b)$, or
\item $w_x > 1$ and $w_y > 1$.
\end{enumerate}
We say $f$ is \textbf{normalised with respect to~$\Delta$} if one of the following holds:
\begin{enumerate}[label=(N\arabic*), ref=N\arabic*, leftmargin=3.5em]
\item $w_x = 0$ or $w_y = 0$,
\item $w_x = 1$ and $w_y = 1$ and $d \leq \min(a, b)$,
\item $w_x = 1$ and $w_y > 1$ and $d \leq b$,
\item $w_x > 1$ and $w_y = 1$ and $d \leq a$, or
\item $w_x > 1$ and $w_y > 1$.
\end{enumerate}
\end{definition}

\begin{proposition} \label{thm:normalised Newton polyhedron}
Every formal power series $f \in \mathbb C[[x, y]]$ is formally right equivalent to a formal power series $g \in \mathbb C[[x, y]]$ which is normalised with respect to each compact facet of its Newton polyhedron. If $f \in \{x, y\}$ defines an isolated singularity, then there exists a right equivalence between $f$ and a polynomial $g \in [x, y]$ which is normalised with respect to each compact facet of its Newton polyhedron.
\end{proposition}

\begin{proof}
We get a formal power series by lines 1--28 and 43 of \cite[Algorithm~4]{BMP20} (skipping lines 29--42). If $f$ defines an isolated singularity, then $f$ is $(\mu+1)$-determined (\cite[Corollary~2.24]{GLS07}) where $\mu$ is its local Milnor number at the origin. Therefore, we have a right equivalence to a polynomial $g$ of degree $\mu+1$ which is normalised with respect to each compact facet of its Newton polyhedron.
\end{proof}

\begin{remark} \label{rem:if equiv to NND}
\begin{enumerate}[label=(\alph*), ref=\alph*]
\item By \cite[Algorithm~4]{BMP20} lines 1--28 and~43, if $f \in \mathbb C\{x, y\}$ is right equivalent to a Newton non-degenerate power series and $f$ is normalised with respect to each compact facet of its Newton polyhedron, then $f$ is Newton non-degenerate.
\item \label{itm:unique polyhedron} By \cite[Remarks 3.27 and~3.28]{BMP20}, if $f \in \mathbb C\{x, y\}$ is right equivalent to a Newton non-degenerate power series, then we can associate a unique, up to reflection, polyhedron to the right equivalence class of~$f$, namely the Newton polyhedron of any power series in the right equivalence class of $f$ which is normalised with respect to each compact facet of its Newton polyhedron.
\end{enumerate}
\end{remark}

If $f \in \mathbb C\{x, y\}$ is not right equivalent to a Newton non-degenerate power series, then we do not have such uniqueness as in \cref{rem:if equiv to NND}\labelcref{itm:unique polyhedron}:

\begin{example} \label{exa:two normalised Newton polyhedrons}
The right equivalent polynomials
\[
\begin{aligned}
  f & := x^2 y^2 (x+y)^2 + x^9 + y^7,\\
  g & := x^2 y^2 (x+y)^2 + x^9 - (x+y)^7
\end{aligned}
\]
are both normalised with respect to each facet of their Newton polyhedrons, but the Newton polyhedron of $f$ is strictly contained in the Newton polyhedron of~$g$. The local Milnor number of $f$ at the origin is $30$ and by \cref{thm:reading lct} the log canonical threshold of $f$ at the origin is $1/3$.
\end{example}

\begin{figure}[h]
  \centering
  \begin{tikzpicture}[scale=0.5]
    \coordinate (up) at (0, 0.08);
    \draw[->] (0,0) -- (10,0);
    \draw[->] (0,0) -- (0,8);
    \draw[blue, thick] ($(0,7)+(up)$) -- ($(2,4)+(up)$) -- ($(4,2)+(up)$) -- (9,0);
    \draw[red, thick] (0,7) -- (2,4) -- (4,2) -- (7,0);
    \draw (0,7) node[right] {$(0,7)$};
    \draw (2,4) node[above right] {$(2,4)$};
    \draw (4,2) node[above right] {$(4,2)$};
    \draw (7,0) node[below] {$(7,0)$};
    \draw (9,0) node[above right] {$(9,0)$};
    \draw (0,8) node[below left] {$y$};
    \draw (10,0) node[below left] {$x$};
    \node[draw, circle, inner sep=1pt, fill] at (0,7) {};
    \node[draw, circle, inner sep=1pt, fill] at (2,4) {};
    \node[draw, circle, inner sep=1pt, fill] at (4,2) {};
    \node[draw, circle, inner sep=1pt, fill] at (7,0) {};
    \node[draw, circle, inner sep=1pt, fill] at (9,0) {};
    \matrix [draw,below left] at (current bounding box.north east) {
      \node [draw=blue, shape=circle, fill=blue, inner sep=1pt, label=right : $f$] {}; \\
      \node [draw=red, shape=circle, fill=red, inner sep=1pt, label=right : $g$] {}; \\
    };
  \end{tikzpicture}
  \caption{Compact faces of the Newton polyhedrons in \cref{exa:two normalised Newton polyhedrons}.}
\end{figure}
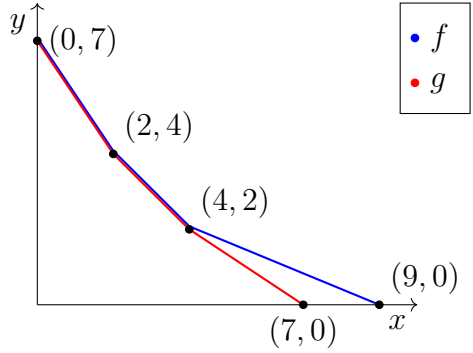

With more care, we can also construct an example where the Newton polyhedrons are not contained in one another:

\begin{example} \label{exa:minimal Newton polyhedrons}
Let the set of polyhedrons in $\mathbb R^2$ have a partial order ``$\leq$'' given by $N_1 \leq N_2$ if and only if $N_1 \subseteq N_2$ or $r(N_1) \subseteq N_2$, where $r$ is the reflection with respect to the line passing through the origin and $(1,1)$. Define
\[
f := \bigl(x y (x + y)\bigr)^7 + (x y)^4 (x + y)^6 (x^8 + y^8) + x y (x^{22} + y^{22}).
\]
Let $\Phi$ be the $\mathbb C$-algebra automorphism of $\mathbb C\{x, y\}$ given by $\Phi\colon x \mapsto x$, $y \mapsto -x-y$. Then $f$ and $\Phi(f)$ are both normalised with respect to each compact facet of their Newton polyhedrons but the Newton polyhedrons of $f$ and $\Phi(f)$ are incomparable. Moreover, the Newton polyhedrons of $f$ and $\Phi(f)$ are precisely all the minimal Newton polyhedrons, up to reflection, in the formal right equivalence class of~$f$. The local Milnor number of $f$ at the origin is $454$ and by \cref{thm:reading lct} the log canonical threshold of $f$ at the origin is $2/21$.
\end{example}

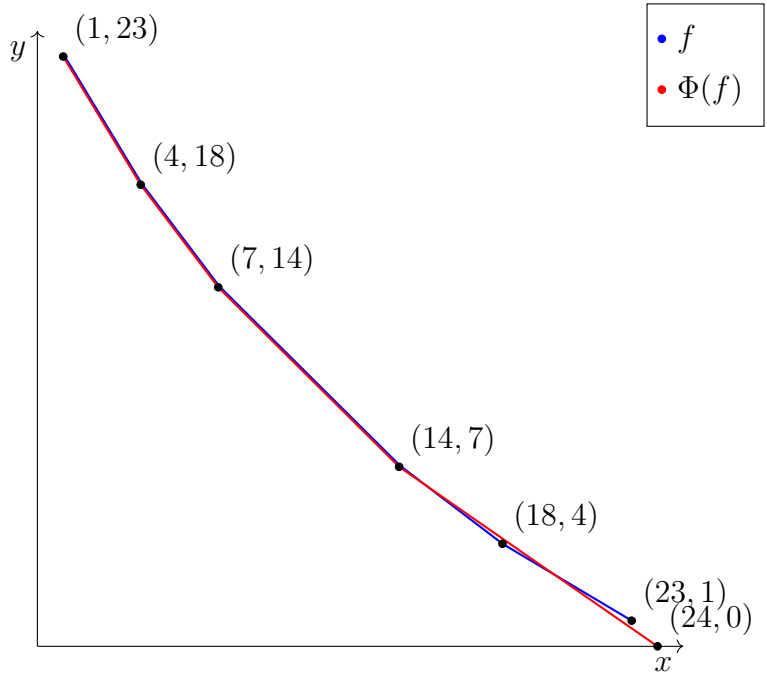
\begin{figure}[h]
  \centering
  \begin{tikzpicture}[scale=0.34]
    \coordinate (up) at (0.04, 0.04);
    \draw[->] (0,0) -- (0,24);
    \draw[->] (0,0) -- (25,0);
    \draw[blue, thick] ($(1,23)+(up)$) -- ($(4,18)+(up)$) -- ($(7,14)+(up)$) -- ($(14,7)+(up)$) -- (18,4) -- (23,1);
    \draw[red, thick] (1,23) -- (4,18) -- (7,14) -- (14,7) -- (24,0);
    \draw (1,23) node[above right] {$(1,23)$};
    \draw (4,18) node[above right] {$(4,18)$};
    \draw (7,14) node[above right] {$(7,14)$};
    \draw (14,7) node[above right] {$(14,7)$};
    \draw (18,4) node[above right] {$(18,4)$};
    \draw (23,1) node[above right] {$(23,1)$};
    \draw (24,0) node[above right] {$(24,0)$};
    \draw (0,24) node[below left] {$y$};
    \draw (25,0) node[below left] {$x$};
    \node[draw, circle, inner sep=1pt, fill] at (1,23) {};
    \node[draw, circle, inner sep=1pt, fill] at (4,18) {};
    \node[draw, circle, inner sep=1pt, fill] at (7,14) {};
    \node[draw, circle, inner sep=1pt, fill] at (14,7) {};
    \node[draw, circle, inner sep=1pt, fill] at (18,4) {};
    \node[draw, circle, inner sep=1pt, fill] at (23,1) {};
    \node[draw, circle, inner sep=1pt, fill] at (24,0) {};
    \matrix [draw,below left] at (current bounding box.north east) {
      \node [draw=blue, shape=circle, fill=blue, inner sep=1pt, label=right : $f$] {}; \\
      \node [draw=red, shape=circle, fill=red, inner sep=1pt, label=right : $\Phi(f)$] {}; \\
    };
  \end{tikzpicture}
  \caption{Compact faces of the Newton polyhedrons in \cref{exa:minimal Newton polyhedrons}.}
\end{figure}

\subsection{Reading the log canonical threshold from the Newton polyhedron} \label{sec:lct from Newton polyhedron}

\begin{theorem} \label{thm:reading lct}
\MainTheorem{}
\end{theorem}

\begin{proof}
Let $\Delta$ be any (not necessarily compact) facet of the Newton polyhedron of $f$ that contains the point $(c, \ldots, c)$. Let $\bm w \in \mathbb Z_{\geq0}^n$ be the unique normal vector of~$\Delta$ such that $\gcd(w_1, \ldots, w_n) = 1$. Let $f_{\bm w}$ denote the $\bm w$-weighted-homogeneous leading term of~$f$. Then, $\operatorname{wt}_{\bm w}(f) / \sum w_i = c$. By \cref{thm:weighted formula power series}, $\operatorname{lct}_{\bm 0}(f) \leq 1/c$.

Now, let $n = 2$ and let $f$ be weakly normalised with respect to~$\Delta$. We use \cref{thm:weighted formula power series} to prove that $\operatorname{lct}_{\bm 0}(f) = 1/c$. For this, we need to show that the pair $(\mathbb C^2, \frac{1}{c} V(f_{\bm w}))$ is log canonical outside the origin, or equivalently that all the irreducible components of $V(f_{\bm w})$ have multiplicity less than or equal to~$c$.

If $w_2 = 0$, then $f_{\bm w} = u x_1^c x_2^k$ for some unit $u \in \mathbb C\{x_2\}$ and non-negative integer~$k \leq c$, showing that all the irreducible components of $V(f_{\bm w})$ have multiplicity less than or equal to~$c$. Similarly in the case $w_1 = 0$.

Below, we consider the case where $w_1$ and $w_2$ are both positive. Let $(P_1, P_2)$ and $(Q_1, Q_2)$ be the two extreme points of~$\Delta$, where $P_1 < Q_1$ and $Q_2 < P_2$. Since $\Delta$ contains the point $(c, c)$, we necessarily have $P_1 \leq c$ and $Q_2 \leq c$. Let $d$ be the greatest integer such that an irreducible power series $g$ to the power $d$ divides the saturation $\operatorname{sat}(f_{\bm w})$ of~$f_{\bm w}$. To show that all the irreducible components of $V(f_{\bm w})$ have multiplicity less than or equal to~$c$, it suffices to prove that $d \leq c$.

If $w_1 = 1$ or $w_2 = 1$, then since $f$ is weakly normalised with respect to~$\Delta$, we have $d \leq \max(P_1, Q_2) \leq c$. Otherwise, after possibly permuting $x_1$ and $x_2$, we have $1 < w_1 < w_2$. Therefore, $\deg g(0, x_2) \geq 2$. Since $g^d \mid \operatorname{sat}(f)$, we find that
\[
g(0, x_2)^d \mid \operatorname{sat}(f_{\bm w})(0, x_2) = x_2^{P_2 - Q_2}.
\]
Therefore, $d \leq (P_2 - Q_2)/2$. Since $\Delta$ has slope $-w_1/w_2 > -1$ and $\Delta$ contains the point $(c, c)$, the line containing $\Delta$ contains the point $(0, \alpha)$ where $\alpha < 2c$. Therefore, $P_2 < 2c$. This proves that $d < c$.

Finally, let $n = 2$ and let $(c, c)$ be contained in the intersection of two facets. Let $(w_1, w_2) \in \mathbb Z_{>0}^2$ be any vector such that $f_{\bm w}$ is the monomial $x^c y^c$. Then $\operatorname{wt}_{\bm w}(f) / \sum w_i = c$. By \cref{thm:weighted formula power series}, since the pair $(\mathbb C^2, \frac{1}{c} V(x^c y^c))$ is log canonical, we find $\operatorname{lct}_{\bm 0}(f) = 1/c$.
\end{proof}

\begin{remark} \label{rem:algorithm}
\begin{enumerate}[label=(\alph*), ref=\alph*]
\item By \cref{thm:normalised Newton polyhedron}, \cref{thm:reading lct} gives an algorithm for computing the log canonical threshold for every non-zero non-unit power series $f \in \mathbb C\{x, y\}$ that defines an isolated singularity. The algorithm outputs positive integers $a, b$ such that the log canonical threshold is~$a/b$.
\item \label{itm:algorithm up to arbitrary precision} By \cref{thm:approximation,thm:normalised Newton polyhedron,thm:reading lct}, if $f$ is formally right equivalent to a power series $g \in \mathbb C[[x, y]]$ such that $g$ is normalised with respect to a facet of the Newton polyhedron that contains the point $(c, c)$, then the log canonical threshold of $f$ is~$1/c$. In particular, this gives an algorithm for computing the log canonical threshold for every non-zero power series $f \in \mathbb C\{x, y\}$ arbitrarily precisely, but it is not necessarily able to find two positive integers such that the log canonical threshold is their quotient.
\item The above in \labelcref{itm:algorithm up to arbitrary precision} is the best possible in the following sense: given two computable convergent power series (meaning where we can compute the terms up to arbitrarily high order), there exists no algorithm to compute two positive integers such that the log canonical threshold is their quotient. This is not surprising, since there is even no algorithm to determine whether a computable power series is equal to zero.
\item In case $f$ is a polynomial, it would be desirable to know a bound on the denominator of the log canonical threshold, perhaps depending only on the degree. By \labelcref{itm:algorithm up to arbitrary precision}, this would give an algorithm for computing the exact value of the log canonical threshold, as a quotient of two positive integers.
\end{enumerate}
\end{remark}

\begin{remark}
The number $c$ in \cref{thm:reading lct} is the minimum of the numbers $\operatorname{wt}_{\bm w'}(f) / \sum w_i'$ taken over all facets of the Newton polyhedron.

More precisely: let $\Delta'$ be any (not necessarily compact) facet of the Newton polyhedron of a non-zero power series $f \in \{x_1, \ldots, x_n\}$ satisfying $f(\bm 0) = 0$. Let $\bm w' = (w_1', \ldots, w_n') \in \mathbb Q_{\geq0}^n$ be any non-negative numbers such that $\bm w'$ is a non-zero normal vector to~$\Delta'$. Then, $\operatorname{wt}_{\bm w'}(f) / \sum w_i' = c'$, where $c' \in \mathbb Q_{\geq0}$ is the unique rational number such that the hyperplane containing $\Delta'$ contains the point $(c', \ldots, c')$. In the notation of \cref{thm:reading lct}, we have $c' \leq c$, with equality if and only if the ray from the origin through the point $(1, \ldots, 1)$ intersects~$\Delta'$.
\end{remark}

\subsection{Example --- sum and product of powers} \label{sec:examples}

As an application of \cref{thm:weighted formula power series}, we compute the log canonical thresholds of power series of the form $(\prod x_i^{a_i}) (\sum x_i^{b_i})$.

\begin{remark}
\Cref{thm:examples power series} is the corrected version of \cite[Example~8.17]{Kol97} and \cite[Proposition~2.2]{Kuw99}, where we have added that the log canonical threshold is at most~$1$. \Cref{exa:counter-example smooth,exa:counter-example Ak} show that this correction is indeed needed.
\end{remark}

\begin{proposition} \label{thm:examples power series}
Let $n \geq 2$. Let $f = (\prod x_i^{a_i}) (\sum x_i^{b_i})$, where $a_1, \ldots, a_n$ are non-negative integers and $b_1, \ldots, b_n$ are positive integers. Then,
\[
\operatorname{lct}_{\bm 0}(f) = \min\mleft(\mleft\{1, \frac{\sum 1/b_i}{1 + \sum a_i/b_i}\mright\} \cup \mleft\{ \frac{1}{a_j} \;\middle|\; j \in \{1, \ldots, n\},\, a_j > 0 \mright\}\mright).
\]
\end{proposition}

\begin{proof}
Denote
\[
M := \min\mleft(\mleft\{\frac{\sum 1/b_i}{1 + \sum a_i/b_i}\mright\} \cup \mleft\{ \frac{1}{a_j} \;\middle|\; j \in \{1, \ldots, n\},\, a_j > 0 \mright\}\mright).
\]

If $M > 1$, then replace $a_n$ by~$1$.

If $M = (\sum 1/b_i)/(1 + \sum a_i/b_i)$, then by \cref{thm:weighted formula power series}, it suffices to show that $(\mathbb C^n, M \cdot V(f))$ is log canonical outside the origin. For every point $P \in \mathbb C^n$, if the $j$-th coordinate $P_j$ of $P$ is non-zero, then after a suitable biholomorphism locally around~$P$, $V(f)$ becomes either $D_1 := V((x_j - P_j) \cdot \prod_{i \neq j} x_i^{a_i})$ or $D_2 := V(\prod_{i \neq j} x_i^{a_i})$. Since the support of both $D_1$ and $D_2$ is snc and since $M \leq 1/a_i$ for all~$i$, we find that $(\mathbb C^n, M \cdot V(f))$ is log canonical at~$P$.

If $M = 1/a_1$ and $n = 2$, then $a_1 \leq a_2 + b_2$.
Therefore, $(\mathbb C^2, M \cdot V(x_1^{a_1} x_2^{a_2 + b_2}))$ is log canonical outside $V(x_1)$. By \cref{thm:weighted formula power series}, $\operatorname{lct}_{\bm 0}(f) = M$. The case where $M = 1/a_2$ and $n = 2$ is similar.

Lastly, if $M = 1/a_j$ and $n \geq 3$, then let $\bm w$ be the non-negative weights where $w_i > 0$ if and only if $i = j$. We have $f_{\bm w} = (\sum_i x_i^{a_i}) (\sum_{i \neq j} x_i^{b_i})$. By \cref{thm:weighted formula power series}, it suffices to show that the pair $(\mathbb C^n, M \cdot V(f_{\bm w}))$ is log canonical outside $V(x_j)$. For this, it suffices to show that the pair
\[
\Bigl(\mathbb C^{n-1},\, M \cdot V\bigl( \bigl(\prod_{i \neq j} x_i^{a_i}\bigr) \bigl(\sum_{i \neq j} x_i^{b_i}\bigr) \bigr)\Bigr)
\]
is log canonical, where both the product and the sum is over $i \in \{1, \ldots, n\} \setminus \{j\}$. The equality $M = 1/a_j$ implies that
\[
M \leq \frac{\sum_{j\neq i} 1/b_i}{c + \sum_{j\neq i} a_i/b_i}.
\]
Log canonicity is now proved by induction over~$n$.
\end{proof}

\newcommand{\etalchar}[1]{$^{#1}$}
\providecommand{\bysame}{\leavevmode\hbox to3em{\hrulefill}\thinspace}
\providecommand{\MR}{\relax\ifhmode\unskip\space\fi MR }
\providecommand{\MRhref}[2]{%
  \href{http://www.ams.org/mathscinet-getitem?mr=#1}{#2}
}
\providecommand{\href}[2]{#2}

\ShowAffiliations{\vspace{0.5\baselineskip},\\[1\baselineskip]}

\end{document}